\documentclass[12pt]{amsart}
\pdfoutput=1

\usepackage{amsmath,amssymb} 
\usepackage{geometry}
\usepackage[english]{babel}
\usepackage[latin1]{inputenc}
\usepackage[all]{xy}
\usepackage{hyperref}

\theoremstyle{plain}
\newtheorem{thm}{Theorem}[section]
\newtheorem{cor}[thm]{Corollary}
\newtheorem{prop}[thm]{Proposition}

\theoremstyle{definition}
\newtheorem{dft}[thm]{Definition}

\theoremstyle{remark}
\newtheorem{rmk}[thm]{Remark}

\newcommand{\N}{\mathbb{N}}
\newcommand{\Z}{\mathbb{Z}}
\newcommand{\R}{\mathbb{R}}
\newcommand{\Lag}{\mathcal{L}}
\newcommand{\parcial}[2]{\frac{\partial #1}{\partial #2}}
\newcommand{\EL}{\mathcal{E}}
\newcommand{\V}{\mathcal{V}} 
\newcommand{\alglie}{\mathfrak{g}}
\newcommand{\fr}{\operatorname{fr }}
\newcommand{\interior}{\operatorname{int }}
\newcommand{\fibG}{\mathfrak{G}^*}
\newcommand{\lag}{\mathit{l}}
\newcommand{\U}{\mathcal{U}}

\begin{document}

\title{Discrete Lagrange problems with constraints valued in a Lie group}

\author{P. M. Chac\'on}
\address{IUFFyM and Dpto. Matem\'aticas, Universidad de Salamanca, Plaza de la Merced 1-4, 37008 Salamanca, Spain.}
\email{pmchacon@usal.es}

\author{A. Fern\'andez}
\address{IUFFyM and Dpto. Matem\'atica Aplicada, Universidad de Salamanca, Casas del Parque 2, 37008 Salamanca, Spain.}
\email{anton@usal.es}

\author{P. L. Garc\'{\i}a} 
\address{IUFFyM-USAL and Real Academia de Ciencias, Plaza de la Merced 1-4, 37008 Salamanca, Spain.}
\email{pgarcia@usal.es}

\keywords{
Cellular complexes, Discrete Lagrange problems, Constraints valued in a Lie group, Euler-Poincar\'e reduction in discrete field theory.
}

\subjclass[2020]{Primary 58A20, 53A70; Secondary 57R15, 49N99, 53Z05, 58E30.}

\thanks{Partially supported by Consejer{\'\i}a de Educaci\'on, Junta de Castilla y Le\'on (Spain), grant SA090G19.}

\begin{abstract}

The Lagrange problem is established in the discrete field theory subject  to constraints with values in a Lie group. For the admissible sections that satisfy a certain regularity condition, we prove that the critical sections of such problems are the solutions of a canonically unconstrained variational problem associated with the Lagrange problem (discrete Lagrange multiplier rule). This variational problem has a discrete Cartan 1-form, from which a Noether theory of symmetries and a multisymplectic form formula are established. The whole theory is applied to the Euler-Poincar\'e reduction in the discrete field theory, concluding as an illustration with the remarkable example of the harmonic maps of the discrete plane in the Lie group $SO(n)$.

\end{abstract}

\maketitle

%\linenumbers

\section{Introduction}

In \cite{lagdisc-jgp} the Lagrange problem is posed and solved in the discrete field theory for constraints valued in a vector space.

More specifically, with the notations and concepts of \cite{lagdisc-jgp} (see sections 2 and 3), the starting point of the doctrine is a bundle $\pi:Y\to V_0$ over the set $V_0$ of the vertices of an arbitrary $n$-dimensional cellular complex $V$ with  fibers $Y_v$, $v \in V_0$, differentiable manifolds  of the same dimension $m$, a Lagrangian density $\Lag : J^1Y\to \R$ on the bundle $j^1\pi:J^1Y \to V_n$ of the 1-jets of $\pi:Y\to V_0$ over the set $V_n$ of the $n$-dimensional cells (faces) of the  complex $V$, and a constraint submanifold  $S= \Phi^{-1}(0) \subset J^1Y$ where $\Phi: J^1Y\to E$ is a differentiable mapping from $J^1Y$ to a real vector space $E$ of dimension $l \le m$. If $\V$ is a finite set of faces of $V$ and $\V_0= \{v\in V_0\,\vert\, v\prec\alpha\in \V\}$ is the (finite) set  of the adherent vertices to the elements of $\V$, a section $y\in \Gamma(\V_0,Y)$ is said to be \textit{admissible}  when $\textrm{img}(j^1 y) \subset S$ ($j^1 y$  the 1-jet extension of the section $y$), and an  infinitesimal variation $\delta y \in T_y\big(\Gamma(\V_0,Y)\big)$ is said to be  \textit{admissible} when $j^1(\delta y)$ is tangent to $S$ along $\textrm{img}(j^1y) \subset  S$ ($j^1(\delta y)$ the 1-jet extension of the infinitesimal variation $\delta y$).  Under these conditions, the aim of the Lagrange problem  is to find the admissible sections that are critical sections of the action functional:
$$
\mathcal{A}_\V (\Lag): y\in \Gamma\big( \V_0,Y) \mapsto \sum_{\alpha \in \V} \Lag \big( (j^1y)(\alpha)\big)
$$
respect to  the admissible infinitesimal variations that vanish at the frontier of $\V$.

Under a certain condition of regularity, it can be proved that these critical sections are the solutions of the unconstrained variational problem with Lagrangian density $\widehat{\Lag}= \Lag + \lambda \circ \Phi$ on the fibered product $J^1Y\times_{V_n}(E^*\times V_n)$,  ($E^*$ dual of $E$), where $\lambda$  is the tautological function $\lambda(\omega, \alpha)\in E^*\times V_n \mapsto \omega \in E^*$ and $\circ$ is the bilinear product given by the duality pairing. This problem has a discrete Cartan 1-form  from which a Noether theory of symmetries and a  multisymplectic form  formula can be established.

Having established this multisymplectic formulation of discrete Lagrangian problems with constraints valued in a vector space, it will be desirable to generalize it to other more complex situations of interest. This is the aim of the present paper, in which we generalize the theory of discrete Lagrangian problems with constraint submanifold $S=\Phi^{-1}(e)\subset J^1Y$, where $\Phi$ is a mapping from $J^1Y$ to a Lie group $G$ ($e\in G$ the identity element of the group). We apply the doctrine thus developed to the important case of the Euler-Poincar\'e reduction in principal bundles by the action of its structural group. See \cite{cc-integradores, marco, reduccion-cont, reduccion-cont-sub, vank} and  references cited there. The interest of this situation continues in a recent pre-print \cite{navier} where a discretization of the Navier-Stokes-Fourier system is developed.

This paper is structured as follows. In Section \ref{sec:lig-grupo-lie} we set out the framework of the theory following the guidelines used in \cite{lagdisc-jgp}. The main result of this section is Theorem \ref{thm:carac-critica-lag-prob} which characterizes critical admissible sections that satisfy an adequate condition of regularity. A key concept of this formulation is the version we give of the Cartan 1-form $\Theta(\Phi)$  associated with the constraint map $\Phi: J^1Y\to G$  via the 1-form over $J^1Y$ with values in the Lie  algebra $\alglie$  of $G$ resulting from the composition $\theta\circ d\Phi$ of the tangent map $d\Phi$ with the Maurer-Cartan  1-form of the group.

In Section \ref{sec:multiplicadores}, the above result allows us identify the critical sections of these problems with the solutions of the Euler-Lagrange equations of an extended unconstrained discrete variational problem adapted to the new class of constraints. This variational problem has a discrete Cartan 1-form, from which a Noether theory of symmetries and a multisymplectic form formula are established.

In Section \ref{sec:euler-poincare} the formalism thus developed is applied to the Euler-Poincar\'e reduction in the discrete field theory, concluding by way of illustration with the example of the harmonic maps of the discrete plane in the Lie group $SO(n)$, which is discussed in Section \ref{sec:ejemplo}.

\section{Discrete Lagrange problems with constraints valued in a Lie group}\label{sec:lig-grupo-lie}
The data that define this type of problems are: a cellular complex, a fiber bundle over the set of vertices, and the jet extensions, a Lagrangian density, and a constraint map. Following the notation of \cite{lagdisc-jgp} (sections 2 and 3), we describe briefly these objects next.
 
\begin{itemize}
\item  For  an $n$-dimensional abstract cellular complex $V$,  $n \in \N$, we denote by $V_k = \{\alpha \in  V \vert \dim \alpha = k\}$ the set of the $k$-dimensional cells, $k\in \N$. In particular, the elements of $V_0$ are called vertices and the elements of $V_n$ are called faces.

We said that a cell $\alpha \in V_k$, $k\in\N$,  is \textbf{adherent} to another cell $\beta \in  V_{k+l}$, $l\in \N$, denoted by $\alpha\prec \beta$ , if $\alpha=\beta$ or there exists a sequence of cells $\alpha = \gamma_k , \gamma_{k+1} ,\dots , \gamma_{k+l} = \beta$, where $\gamma_{k+i}\in V_{k+i}, i=0,\dots, l$, such that each one is incident to the following.

The notion of proximity is given by the \textbf{spherical neighborhood} of a vertex $v\in V_0$, $S_v=\{\alpha\in V_n\vert v\prec \alpha\}$, that is, the set of the different faces that have $v$ as adherent vertex.

Given a vertex $v \in V_0$ and a subset   of faces $\V \subset V_n$, we say that $v$ is \textbf{interior} to $\V$ if $S_v \subset \V$, is exterior to $\V$ if $S_v \subset V_n \setminus \V$ or is \textbf{frontier} of $\V$ otherwise. We denote by  $\interior \V$ and $\fr \V$ the set of interior and frontier vertices of $\V$, respectively.

\item Consider a fiber bundle $\pi:Y\to V_0$ over the set of vertices of an $n$-dimensional cellular complex $V$. That is, $Y$ is a differential manifold, and for any $v\in V_0$ the fiber $Y_v=\pi^{-1}(v)$ is a differentiable manifold of dimension $m\in \N$.

\item  The role of the \textbf{bundle of 1-jets} in the continuous case is given in the discrete theory of fields by the bundle $j^1\pi:J^1Y\to V_n$ over the set of faces of the cellular complex $V$. For that, given $\alpha\in V_n$ consider 
$$(J^1 Y)_\alpha= (j^1\pi)^{-1}(\alpha)=\prod_{v\prec \alpha} Y_v \quad \textrm{ and }\quad J^1Y=\bigsqcup_{\alpha\in V_n}(J^1 Y)_\alpha.$$

For a section $y \in \Gamma(V_0 , Y )$,  the 1-jet extension $j^1 y \in \Gamma(V_n, J^1 Y )$ is defined at  $\alpha\in V_n$ by $(j^1 y)(\alpha ) = (y(v))$, where 
$v \prec \alpha $ .

In this context, the \textbf{2-jet bundle} extension of $\pi : Y \to V_0$ is the bundle $j^2\pi : J^2Y \to V_0$ over the vertices of $V$ such that the fiber over a $v \in V_0$ is:
$$ (J^2Y)_v = (j^2\pi)^{-1}(v)= \prod_{v'\prec \alpha\in S_v}Y_{v'},$$
where $S_v \subset V_n$ is the spherical neighborhood of $v$. That is, the fiber over the vertex $v$ is the product of the fibers of  $\pi:Y\to V_0$ over each vertex of each face that has $v$ as adherent vertex. If $y \in \Gamma (V_0, Y)$ is a section, the 2-jet extension of $y$ is  $j^2y \in \Gamma (V_0, J^2Y)$ defined by $(j^2y)(v) =(
y(v') )$, $v' \prec\alpha\in S_v$ for each vertex $v\in V_0$.

Note that there exists a canonical projection between these jet extensions. More precisely,  for any vertex $v \in V_0$ and any face $\alpha \in S_v$, given $(y_{v'})\in (J^2Y)_v$, where $v'\prec\beta \in S_v$,  we can collect the values over the vertices that are adherent precisely to the face $\alpha$ (and ignore the others). So we can define $\pi_{v\alpha}:(J^2Y)_v\to (J^1Y)_\alpha$ given by 
$$
\pi_{v\alpha}\big(\underbrace{(y_{v'})}_{v'\prec\beta \in S_v} \big)= \big(\underbrace{(y_{v'})}_{v'\prec\alpha}\big).
$$

\item A \textbf{discrete Lagrangian density} $\Lag:J^1Y \to \R$ over the bundle $j^1\pi:J^1Y \to V_n$ is a family of smooth functions $\Lag_\alpha : (J^1Y)_\alpha \to \R$ where $\alpha\in V_n$.

\item Finally, a \textbf{constraint with values in a Lie group} $G$ is a map $\Phi:J^1Y \to G$ such that, for all  $\alpha \in V_n$,  $\Phi_\alpha= \Phi\vert_{(J^1Y)_\alpha}$ is a differential map being $e\in G$ (the identity element of the group) a regular value. Thus, $S_\alpha= \Phi_\alpha^{-1}(e)$ is a submanifold of $(J^1Y)_\alpha$, and we can express $S= \Phi^{-1}(e)$ as the disjoint union $S=\sqcup_{\alpha\in V_n} S_\alpha$. We denote the tangent map of $\Phi_\alpha$, $\alpha\in V_n$, at $j^1_\alpha y \in (J^1Y)_\alpha$ by $(d\Phi_\alpha)_{j^1_\alpha y}$.

\end{itemize}

Now, let $\V$ be a finite set of faces and $\V_0= \{v\in V_0\,\vert\, v\prec \alpha \in \V\}$  the (finite) set of its adherent vertices. Then, similar to those set out in \cite{lagdisc-jgp}, we may give the following definitions.

\begin{dft}\label{def:sec-admisible}
A section $y\in \Gamma(\V_0,Y)$ is said to be admissible if $\textrm{img}(j^1y)\subset S$, that is, $(j^1y)(\alpha)\in S_\alpha \subset (J^1Y)_\alpha$ for all $\alpha \in \V$. 

We shall denote this subset of section by $\Gamma_S(\V_0,Y)$.
\end{dft}

 \begin{dft}\label{def:var-infinitesimal}
 Given an admissible section $y\in \Gamma_S(\mathcal{V}_0,Y)$, an admissible infinitesimal variation of $y$ is a tangent vector $\delta y\in T_y\big(\Gamma(\mathcal{V}_0,Y)\big)$ whose 1-jet extension $j^1\delta y$ is tangent to $S\subset J^1Y$ along $\textrm{img}(j^1y)\subset S$, that is, $(d\Phi_\alpha)_{(j^1y)(\alpha)} ( j^1\delta y)(\alpha)=0$ for all $\alpha \in \mathcal{V}$.
 
%break
We shall denote by $T_y\big(\Gamma_S(\mathcal{V}_0,Y)\big)$ this subspace of infinitesimal variations, and by  \break $T_y^c\big(\Gamma_S(\V_0,Y)\big)$ the subspace of those that vanish at the frontier of  $\mathcal{V}$.
 \end{dft}

\begin{dft}\label{def:sec-critica}
An admissible section $y\in \Gamma_S(\V_0,Y)$ is critical with fixed boundary for the discrete Lagrange problem with Lagrangian density $\Lag:J^1Y \to \R$ and constraint submanifold $S=\Phi^{-1}(e) \subset J^1Y$ if $\big( d\mathcal{A}_\V(\Lag)\big)_{y}=0$ on the subspace $T_y^c\big(\Gamma_S(\V_0,Y)\big)$, where $\mathcal{A}_\V(\Lag)$ is the action functional 
$$\mathcal{A}_\V(\Lag): y\in \Gamma(\V_0,Y)\mapsto \sum_{\alpha\in \V} \Lag\big( (j^1y)(\alpha)\big).$$
\end{dft}

As in \cite{lagdisc-jgp}, under a certain condition of regularity such critical sections can be characterized as the solutions of an Euler-Lagrange operator by proceeding as follows.

From an initial section $y_0\in \Gamma(\V_0,Y)$, let $\Gamma(\V_0,Y)\big(y_0(\fr \V)\big)$ be the submanifold of $\Gamma(\V_0,Y)$ defined by the sections with fixed boundary $y_0(\fr\V)\subset Y$ and $\Gamma_S(\V_0,Y)\big(y_0(\fr\V)\big)$ the subset of the admissible ones.

We define the mapping $\Psi: \Gamma(\V_0,Y)\to \textrm{Map} (\mathcal{V},G)$ by the rule:
\begin{equation}\label{def-psi}
\Psi(y): \alpha\in \V\mapsto \Phi_\alpha \big((j^1y)(\alpha)\big), \quad y \in \Gamma(\V_0,Y).
\end{equation}

Since $\Phi_\alpha:(J^1Y)_\alpha\to G$ is differentiable so is $\Psi$, expressing its differential $(d\Psi)_y:T_y\big( \Gamma(\V_0,Y)\big) \to T_{\Psi(y)}\textrm{Map}(\V,G)$ at $y\in \Gamma(\V_0,Y)$ as:
\begin{equation}\label{deriv-psi}
(d\Psi)_{y}(\delta y): \alpha\in \mathcal{V}\mapsto (d\Phi_\alpha)_{(j^1y)(\alpha)} \big(( j^1\delta y)(\alpha)\big), \quad \delta y \in T_y\big(\Gamma(\V_0,Y)\big).
\end{equation}

Otherwise, taking into account the isomorphism $T_{\Psi(y)}\textrm{Map}(\V,G)\approx \textrm{Map}(\V,\alglie)$ induced by the left translations of the elements $g\in G$, $D_g\in T_gG \mapsto L_{g^{-1}}D_g \in T_eG=\alglie$,  the map $(d\Psi)_y$ can also be expressed as:
\begin{equation}\label{deriv-psi-2}
(d\Psi)_{y}(\delta y): \alpha\in \mathcal{V}\mapsto \big( \theta \circ (d\Phi_\alpha)_{(j^1y)(\alpha)} \big)\big( (j^1\delta y)(\alpha)\big), \quad \delta y \in T_y\big(\Gamma(\V_0,Y)\big),
\end{equation}
where $\theta$ is the Maurer-Cartan 1-form of the group $\theta:g\in G \mapsto \theta_g$ ($\theta_g: D_g\in T_gG\mapsto L_{g^{-1}} D_g\in T_eG= \alglie$) and where $\circ$ is the composition of maps.

In particular, for the admissible sections $y\in \Gamma_S(\V_0,Y)$ both expressions \eqref{deriv-psi} and \eqref{deriv-psi-2} of $(d\Psi)_y$ coincide because in this case $\Phi_\alpha(j^1 y)=e$, $\alpha \in \V$.

By restricting $\Psi$ to the submanifold $\Gamma(V_0,Y) \big(y_0(\fr\V)\big) \subset \Gamma(\V_0,Y)$  and $(d\Psi)_y$  to the subspace  $T_y\big( \Gamma(\V_0,Y)\big(y_0(\fr\V)\big)\big) \subset T_y \big(\Gamma(\V_0,Y)\big)$ we get:

\begin{prop}
$$
\Gamma_S(\V_0,Y) \big(y_0(\fr\V)\big) = \Psi^{-1}(e), \quad T_y^c\big( \Gamma_S(\V_0,Y)\big) = \ker (d\Psi)_y.
$$
\end{prop}
\begin{proof}
It is enough to apply  definitions \ref{def:sec-admisible} and \ref{def:var-infinitesimal}, and formulas \eqref{def-psi} and \eqref{deriv-psi} $\equiv$ \eqref{deriv-psi-2} because $y$ is admissible. 
\end{proof}

Now, a key point of our approach is the following condition of regularity for admissible sections:

\begin{dft}\label{def:cond-regularidad}
An admissible section $y\in \Gamma_S(\V_0,Y) \big(y_0(\fr\V)\big)$ is said to be regular if $(d\Psi)_y : T_y\big( \Gamma(\V_0,Y)\big(y_0(\fr\V)\big)\big) \to \textrm{Map}(\V,\alglie)$ is onto.
\end{dft}

Under this hypothesis, the Inverse Function Theorem proves the existence in the manifold $\Gamma(\V_0,Y)\big(y_0(\fr \V)\big)$ of an open neighborhood $\mathcal{U}(y)$ of $y$ such that $\Gamma_S(\V_0,Y)\big(y_0(\fr \V)\big) \cap \mathcal{U}(y)$ is a submanifold of $\mathcal{U}(y)$ whose tangent space at $y$ is $\ker(d\Psi)_y= T_y^c\big(\Gamma_S(\V_0,Y)\big)$. Therefore, Definition \ref{def:sec-critica} of stationarity is equivalent to stating that $y$ is a critical point of the restriction of the action functional $\mathcal{A}_\V(\Lag)$ to the submanifold $\Gamma_S(\V_0,Y)\big(y_0(\fr \V)\big) \cap \mathcal{U}(y)$. On the other hand, we have the exact sequence: 
\begin{equation}\label{secuencia-exacta}
0\to T_y^ c\big(\Gamma_S(\V_0,Y)\big) \to T_y \Big(\Gamma(\V_0,Y)\big(y_0(\fr\V)\big)\Big) \xrightarrow{(d\Psi)_y}\textrm{Map}(\V,\alglie)\to 0,
\end{equation}
which, finally, will allow us to characterize  critical regular sections as solutions of an Euler-Poincar\'e operator.

Indeed, defining the notion of Cartan 1-form  associated to the constraint morphism $\Phi:J^1Y\to G$ as the family of 1-forms $\Theta_\alpha^v(\Phi)$ on $J^1Y$ with values in Lie algebra $\alglie$, $ \alpha\in \V$, $v\in\V_0$, $v\prec\alpha$, such that
\begin{equation}\label{cartan-phi}
\theta\circ (d\Phi_\alpha)_{j^1_\alpha y} = \sum_{v\prec \alpha} \big( \Theta^v_\alpha (\Phi)\big) _{j^1_\alpha y}, \quad j^1_\alpha y \in (J^1Y)_\alpha,
\end{equation}
the mapping $(d\Psi)_y$ of \eqref{secuencia-exacta} can be expressed as:
\begin{equation}\label{deriv-conmutacion-cartan}
(d\Psi)_{y}(\delta y): \alpha\in \V\mapsto \sum_{v\prec \alpha} \big( \Theta_\alpha^ v (\Phi) \big)_{(j^ 1y)(\alpha)}  (j^1\delta y)(\alpha), \quad \delta y \in T_y\Big(\Gamma(\V_0,Y)\big(y_0(\fr\V)\big)\Big),
\end{equation}
from where we can obtain the following characterization of the critical sections.

\begin{thm}\label{thm:carac-critica-lag-prob}
A regular admissible section $y\in \Gamma_S(\V_0,Y)\big( y_0(\fr\V)\big)$ is critical with fixed boundary for the Lagrange problem with Lagrangian density $\Lag: J^1Y\to \R$ and constraint submanifold $S=\Phi^{-1}(e)\subset J^1Y$ if and only if there exists a mapping $\lambda\in \textrm{Map}(\V,\alglie^*)$ such that for any $v\in\interior\V$:
\begin{equation}\label{multiplicador}
\Big(\EL_v(\Lag) + \sum_{\alpha\in S_v} \pi^*_{v\alpha}\big( \lambda(\alpha)\circ \Theta_\alpha^v (\Phi)\big) \Big)_{(j^2y)(v)}=0,
\end{equation}
where $\EL_v(\Lag)$ is the Euler-Lagrange 1-form of $\Lag$ at $v$ as unconstrained variational problem, $S_v$ is the spherical neighborhood of $v$, $\circ$ is the bilinear product of the duality between $\alglie$ and $\alglie^*$, and $\pi_{v\alpha}$ is the canonical projection from $(J^2Y)_v$ to $(J^1Y)_\alpha$. This mapping $\lambda$ is unique and we shall call the \textbf{multiplier associated} to the critical regular section $y$.
\end{thm}
\begin{proof}
It is enough to repeat mutatis mutandis the proof of the Theorem 4.6 of Section 4 in \cite{lagdisc-jgp}. Indeed, from the exact dual sequence of \eqref{secuencia-exacta}:
$$
0\to \textrm{Map}(\V,\alglie)^* \xrightarrow{(d\Psi)_y^*}   T_y \Big(\Gamma(\V_0,Y)\big(y_0(\fr\V)\big)\Big)^*  \to T_y^c\big(\Gamma_S(\V_0,Y)\big) ^* \to 0,
$$
follows the identification of $\textrm{img}(d\Psi)_y^*$  with the elements of $T_y \Big(\Gamma(\V_0,Y)\big(y_0(\fr\V)\big)\Big)^*$ incidents with $T_{y}^{c}\big(\Gamma_{S}(\V_0,Y)\big)$.

So, if the section $y$ is critical (Definition \ref{def:sec-critica}), then there exists a map $\lambda\in\textrm{Map}(\V,\alglie)^*$ $= \textrm{Map}(\V,\alglie^*)$, that is unique by the injectivity of $(d\Psi)_y^*$, such that:
\begin{equation}\label{existe-lambda}
-\big(d\mathcal{A}_\V(\Lag)\big)_y (\delta y) = \big( (d\Psi)_y^*\lambda\big) (\delta y), \quad \delta y \in T_y \Big(\Gamma(\V_0,Y)\big(y_0(\fr\V)\big)\Big).
\end{equation}

By the variation formula (10) of \cite{lagdisc-jgp}, Section 3:
$$
\big(d\mathcal{A}_\V(\Lag)\big)_y (\delta y) = \sum_{v\in\interior\V} \big(\EL_v(\Lag)\big)_{(j^2y)(v)} (j^2\delta y )(v),
$$
since $(\delta y)_v=0$ at the vertices $v\in\fr\V$.

On the other hand, by \eqref{deriv-conmutacion-cartan} we have:
\begin{align*}
 \big( (d\Psi)_y^*\lambda\big)  (\delta y)  = & \lambda \circ (d\Psi)_y(\delta y) = \sum_{\alpha\in \V} \lambda(\alpha) \circ  \bigg( \sum_{v\prec \alpha} \big( \Theta_\alpha^ v (\Phi) \big)_{(j^ 1y)(\alpha)}  (j^1\delta y)(\alpha)\bigg)\\
= & \sum_{v \in\interior \V} \sum_{\alpha\in S_v} \lambda(\alpha) \circ  \big( \Theta_\alpha^ v (\Phi) \big)_{(j^ 1y)(\alpha)}  (j^1\delta y)(\alpha),
\end{align*}
since by swapping the sums, as $(\delta y)_v=0$ at the vertices $v\in\fr\V$, we are left with only the interior vertices and the faces of the corresponding spherical neighborhood. By replacing it now in equality \eqref{existe-lambda}, and using the canonical projection $\pi_{v\alpha}:(J^2Y)_v\to (J^1Y)_\alpha$, we finally obtain:
$$
\sum_{v\in \interior\V}\Big(\EL_v(\Lag) + \sum_{\alpha\in S_v} \pi^*_{v\alpha}\big( \lambda(\alpha)\circ \Theta_\alpha^v (\Phi)\big) \Big)_{(j^2y)(v)} (j^2\delta y ) (v)=0.
$$

From here the result follows because $(\delta y)_v$ is an arbitrary variation at the vertices $v\in\interior\V$. 

\end{proof}

\section{Lagrange multiplier rule. Noether's theorem. Multisymplectic form formula}\label{sec:multiplicadores}

The aim of section is to formulate the Lagrange problems that we have seen in the previous section as free variational problems via a Lagrange multipliers rule. In this sense the equations \eqref{multiplicador} that characterize the critical sections of the  Lagrange problem will be interpreted as the Euler-Lagrange equations of an unconstrained variational problem extended to the multipliers $\lambda\in \textrm{Map}(\V,\alglie^*)$ as follows.

Let $\alpha\in V_n$ and let us denote by $\Lag_\alpha$ and $\Phi_\alpha$ the restriction of $\Lag$ and $\Phi$ to the fibers $(J^1Y)_\alpha$ as before. Then, for every $\lambda_\alpha\in \alglie^*$ we can define the 1-form
\begin{equation*}
d\Lag_\alpha + \lambda_\alpha \circ (\theta\circ d\Phi_\alpha)
\end{equation*}
on  $(J^1Y)_\alpha$; where the first $\circ$ from the left denotes the duality pairing between $\alglie$ and $\alglie^*$.  That is, for every $\lambda\in \textrm{Map}(V_n,\alglie^*)$ we have the 1-form
\begin{equation}\label{deriv-prob-extendido}
d\Lag + \lambda \circ (\theta\circ d\Phi),
\end{equation}
on $J^1Y \times \alglie^* \equiv  J^1Y\times_{V_n} \fibG$, where $\fibG=\alglie^* \times V_n \to V_n$ is the trivial bundle on $V_n$ with fiber $\alglie^*$. We will use $\lambda_\alpha$ or $\lambda(\alpha)$ to denote the same element of $\alglie^*$.

This 1-form defines an unconstrained variational problem on the manifold $\Gamma(\V_0\times Y)\times \textrm{Map}(\V,\alglie^*)$ with action 1-form given by:
\begin{equation}\label{accion-extendida}
\big( \mathbb{A}_\V (\Lag,\Phi)\big)_{(y,\lambda)} (\delta y,\delta \lambda) = \sum_{\alpha\in \V} (d\Lag_\alpha)_{(j^1y)(\alpha)}(j^1\delta y)(\alpha) + \lambda_\alpha \circ (\theta \circ d\Phi_\alpha)_{(j^1y)(\alpha)} (j^1\delta y)(\alpha),
\end{equation}
$ (y,\lambda)\in \Gamma(\V_0,Y) \times \textrm{Map}(\V,\alglie^*)$, $(\delta y, \delta\lambda) \in T_y \big(\Gamma(\V_0,Y) \big)\oplus T_\lambda\big(\textrm{Map}(\V,\alglie^*)\big)$, from where the following notion of stationarity  is given.

\begin{dft}\label{def:estacionaridad}
% break
We say that a section-mapping $ (y,\lambda)\in \Gamma(\V_0,Y) \times \textrm{Map}(\V,\alglie^*)$ is critical for the action 1-form $ \mathbb{A}_\V (\Lag,\Phi)$ if the section  $ y\hspace{-0.9pt}\in \hspace{-1pt}\Gamma(\V_0,Y) $ is admissible and  $\hspace{-1pt}\big( \mathbb{A}_\V (\Lag,\Phi)\big)_{(y,\lambda)}$ vanishes on the subspace $T_y^c \big(\Gamma(\V_0,Y) \big)\oplus T_\lambda\big(\textrm{Map}(\V,\alglie^*)\big)$.
\end{dft}

From \eqref{accion-extendida} we get that the subspace  $T_\lambda\big(\textrm{Map}(\V,\alglie^*)\big)$ is incident with the 1-form $ \mathbb{A}_\V (\Lag,\Phi)$, i.e.:
\begin{equation}\label{incidente-accion}
\big( \mathbb{A}_\V (\Lag,\Phi)\big)_{(y,\lambda)} (0,\delta \lambda)= 0 \quad \textrm{ for all } \quad  \delta \lambda \in T_\lambda\big(\textrm{Map}(\V,\alglie^*)\big).
\end{equation}

\begin{rmk}\label{rmk:sin-lag-extendido}
The unconstrained variational problem defined by the 1-form \eqref{deriv-prob-extendido} differs substantially from that introduced in Section 5 of  \cite{lagdisc-jgp} for Lagrange problems with constraints $\Phi:J^1Y\to E$ valued in vector spaces. As noted in the Introduction, these problems are defined by an unconstrained Lagrangian density $\widehat{\Lag}= \Lag + \lambda \circ \Phi$ ($\circ$ the bilinear product of the duality between $E$ and $E^*$), which is not generalizable to constraints $\Phi:J^1Y\to G$ valued in Lie groups. However, differentiating $\widehat{\Lag}$ we obtain:
$$
d\widehat{\Lag} = d\Lag + \lambda \circ d\Phi + d\lambda \barwedge\Phi,
$$
where $\barwedge$ denotes the wedge product of valued forms with respect to the bilinear form $\circ$ induced by the duality pairing. The first two terms of this formula are generalizable via the expression \eqref{deriv-prob-extendido} in which the Maurer-Cartan 1-form of   $G$ is inserted. In this sense, it is important to observe that by dispensing with the term $d\lambda \barwedge\Phi$ in this formulation (which is also not generalizable),  condition \eqref{incidente-accion} no longer allows to obtain the admissibility of the critical section-mappings as part of the Euler-Lagrange equations of the extended problem. This condition must be imposed separately, as contemplated in Definition \ref{def:estacionaridad} that we have given of stationarity.
\end{rmk}

Following the route of Section 5 of \cite{lagdisc-jgp} with the obvious changes we will have:
\begin{prop}[Variational formula]\label{prop:formula-variacion}
% break
For any section-mapping $ (y,\lambda)\in \Gamma(\V_0,Y)  \times$ \break $\textrm{Map}(\V,\alglie^*)$ and for any tangent vector $(\delta y, \delta\lambda) \in T_{(y,\lambda)} \big(\Gamma(\V_0,Y) \times \textrm{Map}(\V,\alglie^*)\big)$,  we have \break now:
\begin{equation}\label{formulavariacion-prob-lag}
\begin{aligned}
\big(\mathbb{A}_\V (  \Lag, \Phi  )\big)_{(y,\lambda)} (\delta y, \delta\lambda)  = &  \sum_{v\in\interior\V}\big( \EL_v(\Lag)+ \sum_{\alpha\in S_v} \lambda(\alpha) \circ \pi^*_{v\alpha}\Theta_\alpha^v(\Phi)\big)_{(j^2y)(v)} (j^2\delta y)(v)\\
 & +  \sum_{(v\in \fr\V)\prec \alpha\in \V}\big( \Theta_\alpha^v(\Lag) + \lambda(\alpha) \circ \Theta_\alpha^v(\Phi)\big)_{(j^1y)(\alpha)} (j^1\delta y)(\alpha).
\end{aligned}
\end{equation}
\end{prop}

\begin{cor}[Euler-Lagrange equations]\label{cor:Euler-Lagrange}
A section-mapping  $ (y,\lambda)\in \Gamma(\V_0,Y) \times$ \break $\textrm{Map}(\V,\alglie^*)$ is critical for the action 1-form $\mathbb{A}_\V (  \Lag, \Phi  )$ if and only if the section $y\in \Gamma(\V_0,Y)$ is admissible and for any $v\in \interior \V$:
\begin{equation}\label{EL-prob-ligado}
\big(\EL_v(\Lag)+ \sum_{\alpha\in S_v} \lambda(\alpha) \circ \pi^*_{v\alpha}\Theta_\alpha^v(\Phi)\big)_{(j^2y)(v)} = 0.
\end{equation}
\end{cor}
\begin{thm}[Lagrange Multiplier Rule]
% break
A  section-mapping  $ (y,\lambda) \in \Gamma(\V_0,Y) \times$ \break $\textrm{Map}(\V,\alglie^*)$, being $y$ an admissible regular section, is critical for the action 1-form  $\mathbb{A}_\V (\Lag,\Phi)$ if and only if its component $y\in \Gamma(\V_0,Y)$ is an admissible regular critical section for the Lagrange problem with Lagrangian $\Lag:J^1Y \to \R$ and constraint submanifold $S=\Phi^{-1}(e)\subset J^1Y$, being its component $\lambda$ the multiplier associated to $y$.
\end{thm}

On the other hand, the expression of the boundary term of the variation formula \eqref{formulavariacion-prob-lag} suggests to take as  Cartan 1-form associated to the  action 1-form $\mathbb{A}_\V(\Lag,\phi)$ the family of 1-forms:
$$
\Theta_\alpha^v(\Lag) + \lambda_\alpha \circ \Theta_\alpha^v(\Phi), \quad v\in\V_0, \quad \alpha \in \V_n, \quad  v\prec \alpha.
$$

This concept will allow us to formulate a Noether's theory of  symmetries for the  Lagrange problems with constraints valued in Lie groups, as well as to establish a corresponding formula of the  multisymplectic form.

\begin{dft}
An infinitesimal symmetry of a Lagrange problem with Lagrangian density $\Lag: J^1Y \to \R$ and constraint submanifold $S= \Phi^{-1}(e)$, with $\Phi: J^1 Y \to G$,  is a vector field $D\in \mathfrak{X}(Y)$ such that $(j^1D)\Lag=0$ and $(j^1D)(\Phi)= (d\Phi)(j^1D)=0$ where:
$$
(d\Phi)(j^1 D): j^1_\alpha y\in (J^1Y)_\alpha \mapsto (d\Phi)_{j^1_\alpha y}(j^1D)_{j^1_\alpha y} \in T_{\Phi(j^1_\alpha y)} G . %, \quad j^1_\alpha y \in (J^1Y)_\alpha.
$$
\end{dft}

\begin{thm}[Noether]
If $y\in \Gamma_S(\V_0,Y)$ is a regular critical section of the Lagrange problem given by $(\Lag, \Phi)$ and $D\in \mathfrak{X}(Y)$ is an infinitesimal symmetry, then:
$$
\sum_{(v\prec \fr\V)\prec \alpha} \left[ \big( \Theta^v_\alpha (\Lag) + \lambda (\alpha)\circ \Theta ^v_\alpha (\Phi)\big) ( j^1D)_\alpha \right] (j^1y)(\alpha)= 0,
$$
where $\lambda \in \textrm{Map} (\V,\alglie^*)$ is the multiplier associated to the regular critical section.
\end{thm}
\begin{proof}
If we call $(\delta y)(v)=(D_v)_{y(v)}$, $v\in \V_0$, then $(j^1\delta y)(\alpha)=\big( (j^1D)_\alpha\big)_{(j^1y)(\alpha)}$, $\alpha\in \V$, and we obtain $\big(\mathbb{A}_\V(\Lag,\Phi)\big)_{(y,\lambda)}(\delta y,0)=0$ by \eqref{accion-extendida} and for being $D$ an infinitesimal symmetry of the Lagrange problem. On the other hand, if $y\in \Gamma_S(\V_0,Y)$ is a regular critical section of the  Lagrange problem with associated multiplier  $\lambda\in \textrm{Map}(\V,\alglie^*)$, then for any $v\in \interior\V$ Theorem \ref{thm:carac-critica-lag-prob} provides:
$$
\left( \EL_v(\Lag) + \sum_{\alpha\in S_v} \lambda (\alpha)\circ \pi^*_{v\alpha}\Theta^v_\alpha (\Phi)\right)_{(j^2y)(v)} = 0.
$$

Now the result follows if we substitute these both zero terms in the variation formula \eqref{formulavariacion-prob-lag}. 
\end{proof}

As for the establishment of a multisymplectic form formula  for this kind of Lagrange problems we will proceed as follows.

The variation formula \eqref{formulavariacion-prob-lag} can be expressed on the manifold $\Gamma(\V_0,Y) \times \textrm{Map}(\V,\alglie^*)$ as:
\begin{equation}\label{formulavariacion-prob-lag-sin-evaluar}
\mathbb{A}_\V (  \Lag, \Phi  ) =   \sum_{v\in\interior\V}\big( \EL_v(\Lag)+ \sum_{\alpha\in S_v} \lambda_\alpha \circ \Theta_\alpha^v(\Phi)\big)  +  \sum_{(v\in \fr\V)\prec \alpha\in \V}\big( \Theta_\alpha^v(\Lag) + \lambda_\alpha \circ \Theta_\alpha^v(\Phi)\big),
\end{equation}
%cambio 
where the terms of this sum are considered 1-forms on $\Gamma(\V_0,Y) \times \textrm{Map}(\V,\alglie^*)$ through the canonical projections  from $\Gamma(\V_0,Y) \times \textrm{Map}(\V,\alglie^*)$ to $(J^1Y\times_\V \fibG)_\alpha$ and $(J^2Y\times_{\V_0}j^1\fibG)_v$ respectively, $(y,\lambda) \mapsto \big( (j^1y)(\alpha), e^* (\alpha)=(\alpha,\lambda(\alpha))\big)$ and  $(y,\lambda) \mapsto \big( (j^2y)(v), j^1e^* (v)\big)$, being $j^1\fibG\to \V_0$ the 1-jet bundle of $\fibG \to \V$ defined by $(j^1\fibG)_v= \prod_{\alpha\in S_v} \fibG_\alpha$, $v\in \V_0$, and where $j^1e^*: v\in \V_0\mapsto \big(e^*(\alpha)\big)$, $\alpha\in S_v$, denotes the 1-jet extension of $e^* \in \Gamma(\V, \fibG)$.

If $\EL(\Lag,\Phi)$ is the 1-form on $\Gamma(\V_0,Y) \times \textrm{Map}(\V,\alglie^*)$:
\begin{equation}\label{EL-extendido}
\EL(\Lag, \Phi)= \sum_{v\in \interior\V} \big( \EL_v(\Lag) + \sum_{\alpha\in S_v} (\lambda_\alpha \circ \Theta_\alpha^v )(\Phi)\big)
\end{equation}
then the critical section-mappings $(y,\lambda)\in \Gamma(\V_0,Y) \times \textrm{Map}(\V,\alglie^*)$ are characterized through Corollary \ref{cor:Euler-Lagrange} by the conditions:
$$
y \quad \textrm{  admisible  and }\quad  \big(\EL(\Lag,\Phi)\big)_{(y,\lambda)}=0.
$$

\textbf{Local expression.}  Let $\{y^i_v\vert 1\le i\le m\}$ be a local coordinate system  in a neighborhood of each $y_v\in Y_v$ and $\{y^i_v\vert 1\le i\le m, v\prec \alpha\}$, the local coordinate system  induced in $(J^1Y)_\alpha$.  On the other hand, let $\{g^I\vert 1\le I\le l\}$, $l=\dim G$, be a local system of coordinates of $G$  in a neighborhood of $e\in G$ and $\big\{ \left(\parcial{}{g^I}\right)_e \vert 1\le I\le l\big\}$ and $\big\{ (dg^I)_e\vert 1\le I\le l \big\}$ the induced basis into the Lie algebra $\alglie$ and its dual $\alglie^*$ respectively. If $L_g:G\to G$ is the left translation by $g\in G$, $L_g(\overline{g})=g\cdot\overline{g}$, $L_g^I:G\to \R$ is the function $L_g^I= g^I\circ L_g$ and $\Phi_\alpha^I= g^I\circ \Phi_\alpha$, $1\le I\le l$, then for any $\alpha\in \V_n$, $\theta\circ d\Phi_\alpha$ is the 1-form over $(J^1Y)_\alpha$ with values in the Lie algebra $\alglie$ given by:
\begin{equation}\label{expr-local}
(\theta \circ d\Phi_\alpha)_{j^1_\alpha y} = \sum_{\substack{v\prec\alpha \\ 1\le I \le l \\1\le i \le m}} \left(\sum_{J=1}^l\bigg( \parcial{L^I_{g^{-1}}}{g^J}\bigg)_{g} \bigg(\parcial{\Phi^J_\alpha }{y^i_v}\bigg)_{j^1_\alpha y}\right) (dy^i_v)_{j^1_\alpha y} \otimes \Big(\parcial{}{g^I}\Big)_e
\end{equation}
for $j^1_\alpha y \in (J^1Y)_\alpha$ and $g= \Phi_\alpha(j^1_\alpha y) \in G$.

From here it follows that:
\begin{equation}\label{eulerlagrange-coord-local}
 \EL(\Lag, \Phi) = \sum_{v\in \interior \V}\sum_{i=1}^m \EL^i_v(\Lag, \Phi) dy^i_v ,
 \end{equation}
where, for $\lambda(\alpha)= \sum_{1\le I\le l}\lambda(\alpha)^I(dg^I)_e$:
$$
\EL^i_v(\Lag, \Phi) = \sum_{\alpha\in S_v} \parcial{\Lag_\alpha}{y^i_v} +  \sum_{I,J=1}^l\lambda(\alpha)^I  \bigg( \parcial{L^I_{g^{-1}}}{g^J}\bigg) \bigg(\parcial{\Phi^J_\alpha }{y^i_v}\bigg).
$$

\begin{dft}
% cambio break
A Jacobi field along a critical section-mapping $(y,\lambda)\in \Gamma(\V_0,Y)\times$ \break $\textrm{Map}(\V,\alglie^*)$ of the action 1-form $\mathbb{A}_\V ( \Lag, \Phi )$ is a tangent vector $(\delta y,\delta \lambda) \in T_{(y,\lambda)}\big( \Gamma(\V_0,Y)\times \textrm{Map}(\V,\alglie^*) \big) $ that verifies: $\delta y$ is admissible and $(\delta y,\delta \lambda)$ has an extension $D\in \mathfrak{X}\big( \Gamma(\V_0,Y) \times \textrm{Map}(\V,\alglie^*)\big)$ such that:
\begin{equation}\label{ecuacion-jacobi}
 \big(L_D \EL(\Lag,\Phi)\big)_{(y,\lambda)}=0.
 \end{equation}
\end{dft}

The definition is correct since from \eqref{eulerlagrange-coord-local} we have:
$$
L_D\EL(\Lag,\Phi) = \sum_{v\in\interior\V} (D\EL^i_v) dy^i_v + \EL^i_v(\Lag,\Phi) L_D dy^i_v ,
$$
and therefore, evaluating this expression at $(y,\lambda)$, the second term vanishes because $(y,\lambda)$ is critical and the first one depends only on $(\delta y,\delta \lambda)$.

Continuing with this approach, a key point for obtaining  a multisymplectic form  formula for this new class of Lagrange problems is the following property.

\begin{prop}\label{prop:deriv-accion-ext}
If $(y,\lambda)\in \Gamma(\V_0,Y)\times\textrm{Map}(\V,\alglie^*)$ is a section-mapping with $y$ admissible and $\big( (\delta y)^i, (\delta\lambda)^i\big) \in T_{(y,\lambda)} \big( \Gamma(\V_0,Y)\times\textrm{Map}(\V,\alglie^*)\big) $ are two tangent vectors at $(y,\lambda)$ with $(\delta y)^i$ admissible, $i=1,2$, we have:
$$
\big(d\mathbb{A}_\V (\Lag, \Phi)\Big) _{(y,\lambda)} \big( \big( (\delta y)^1, (\delta\lambda)^1\big), \big( (\delta y)^2, (\delta\lambda)^2\big)\big) = 0.
$$
\end{prop}
\begin{proof}
From \eqref{accion-extendida}:
$$
\mathbb{A}_\V(\Lag,\Phi)=  \sum_{\alpha\in \V} d\Lag_\alpha + \lambda_\alpha \circ (\theta \circ d\Phi_\alpha)
$$
where the terms of this sum are considered 1-forms over $\Gamma(\V_0,Y)\times\textrm{Map}(\V,\alglie^*)$ via the canonical projection  $(y,\lambda) \mapsto \big( (j^1y)(\alpha), e^* (\alpha)=(\alpha,\lambda(\alpha))\big)$ from  $\Gamma(\V_0,Y) \times \textrm{Map}(\V,\alglie^*)$ to $(J^1Y\times_\V \fibG)_\alpha$.

From here we get:
$$
d\mathbb{A}_\V(\Lag,\Phi)=  \sum_{\alpha\in \V} \underbrace{d^2\Lag_\alpha}_{=0} + d\lambda_\alpha \barwedge (\theta \circ d\Phi_\alpha)+ \lambda_\alpha \circ d(\theta \circ d\Phi_\alpha),
$$
where, as before, $\barwedge$ is the wedge product of valued forms in $\alglie^*$ with forms valued in $\alglie$, applying the duality pairing. This gives:
\begin{align*}
\big(d\mathbb{A}_\V(\Lag,\Phi)\big)_{(y,\lambda)} & \big( \big( (\delta y)^1, (\delta\lambda)^1\big), \big( (\delta y)^2, (\delta\lambda)^2\big)\big)   = \\
&  \sum_{\alpha\in \V} \big( d\lambda_\alpha \barwedge (\theta \circ d\Phi_\alpha)\big)_{(y,\lambda)} \big( \big( (\delta y)^1, (\delta\lambda)^1\big), \big( (\delta y)^2, (\delta\lambda)^2\big)\big)  \\
& + \big( \lambda_\alpha \circ d(\theta \circ d\Phi_\alpha)\big)_{(y,\lambda)} \big( \big( (\delta y)^1, (\delta\lambda)^1\big), \big( (\delta y)^2, (\delta\lambda)^2\big)\big) .
\end{align*}

The first term vanishes because $(\delta y)^i$, $i=1,2$ are admissible. Indeed:
\begin{align*}
\sum_{\alpha\in \V}&   \big(  d\lambda_\alpha  \barwedge (\theta \circ d\Phi_\alpha)\big)_{(y,\lambda)} \big( \big( (\delta y)^1, (\delta\lambda)^1\big), \big( (\delta y)^2, (\delta\lambda)^2\big)\big)  =  \\
 \sum_{\alpha\in \V} & \Big(  (d\lambda_\alpha)_{(y,\lambda)} \big( (\delta y)^1, (\delta\lambda)^1\big) \cdot (\theta \circ d\Phi_\alpha))_{(y,\lambda)} \big( (\delta y)^2, (\delta\lambda)^2\big)  \\
& - ( d\lambda_\alpha)_{(y,\lambda)} \big( (\delta y)^2, (\delta\lambda)^2\big) \cdot (\theta \circ d\Phi_\alpha))_{(y,\lambda)} \big( (\delta y)^1, (\delta\lambda)^1\big) \Big) = 0
\end{align*}

since we know from the admissibility of $(\delta y)^i$, by Definition 2.2, that 
$$(d\Phi_\alpha)_{(j^1y)(\alpha)} (j^1\delta y^i)(\alpha) \big) =0,$$
for $i=1,2$, and so:
$$
(\theta \circ d\Phi_\alpha))_{(y,\lambda)} \big( (\delta y)^i, (\delta\lambda)^i\big) = \theta_{\Phi_\alpha((j^1y)(\alpha))} \big( (d\Phi_\alpha)_{(j^1y)(\alpha)} (j^1\delta y^i)(\alpha) \big) =0 \quad i=1,2.
$$

Regarding the second term, from the local expression \eqref{expr-local} we obtain:
\begin{align*}
d(\theta \circ d\Phi_\alpha)  = & \left( d\left( \sum_{v\prec \alpha, I, i} \left(\sum_{J}\bigg( \parcial{L^I_{g^{-1}}}{g^J}\bigg) \bigg(\parcial{\Phi^J_\alpha }{y^i_v}\bigg)\right) \right)\wedge dy^i_v \right)\otimes \Big(\parcial{}{g^I}\Big)_e \\
= &  \sum_{ I,J}  \left( d \bigg( \parcial{L^I_{g^{-1} }}{g^J}\bigg) \wedge d\Phi^J_\alpha +  \bigg( \parcial{L^I_{g^{-1}}}{g^J}\bigg)  \underbrace{d^2\Phi^J_\alpha}_{=0} \right)\otimes \Big(\parcial{}{g^I}\Big)_e 
\end{align*}
from which, if $\lambda(\alpha)= \sum_{I}\lambda(\alpha)^I (dg^I)_e$ as before, we get:
\begin{align*}
\sum_{\alpha\in \V} &  \big( \lambda_\alpha \circ d(\theta \circ d\Phi_\alpha)\big)_{(y,\lambda)} \big( \big( (\delta y)^1, (\delta\lambda)^1\big), \big( (\delta y)^2, (\delta\lambda)^2\big)\big) =\\
\sum_{\alpha\in \V, I,J} & \lambda(\alpha)^I  \left( d \bigg( \parcial{L^I_{g^{-1}}}{g^J}\bigg) \wedge d\Phi^J_\alpha \right)_{(j^1y)(\alpha)}  \big( j^1(\delta y)^1(\alpha) ,j^1(\delta y)^2(\alpha) \big)  =0
\end{align*}
because of, again, the admissibility of  $(\delta y)^i$, $i=1,2$. 
\end{proof}

\begin{rmk}
 
 As noted in Remark \ref{rmk:sin-lag-extendido}, the action 1-form $\mathbb{A}_\V(\Lag,\Phi)$ generalizes to discrete Lagrange problems valued in a Lie group the  differential, $d\mathbb{A}_\V(\widehat{\Lag})$,  of the action function $\mathbb{A}_\V(\widehat{\Lag})$  of the extended Lagrangian $\widehat{\Lag}= \Lag + \lambda \circ \Phi$ of discrete Lagrange problems valued in a vector space. In this case, Proposition \ref{prop:deriv-accion-ext} is trivial because:
 
 $$ d\big(d\mathbb{A}_\V(\widehat{\Lag}) \big) = d^2\mathbb{A}_\V(\widehat{\Lag})=0. $$
\end{rmk}

At this point we finally have:

\begin{thm}[Multisymplectic form formula]
% break
If  $(y,\lambda) \in \Gamma(\V_0,Y) \times \textrm{Map}(\V,\alglie^*)$ is a critical section-mapping  of the action 1-form $\mathbb{A}_\V(\Lag, \Phi)$ and $\big( (\delta y)^{i},(\delta\lambda)^{i}\big)\in  T_{(y,\lambda)}\big( \Gamma(\V_0,Y) \times \textrm{Map}(\V,\alglie^*)\big)$, $i=1,2$,  are two Jacobi fields along $(y,\lambda)$, then we have:
$$
\sum_{(v\prec \fr \V)\prec \alpha} \Big( d \big( \Theta^v_\alpha (\Lag) + \lambda_\alpha \circ \Theta ^v_ \alpha (\Phi)\big) \Big)_{(y,\lambda)} \big( \big( (\delta y)^1, (\delta\lambda)^1\big), \big( (\delta y)^2, (\delta\lambda)^2\big)\big) = 0.
$$
\end{thm}
\begin{proof}
Taking the differential of \eqref{formulavariacion-prob-lag-sin-evaluar} and having in mind \eqref{EL-extendido}, we get:
$$
d\mathbb{A}_\V(\Lag,\Phi) = d \EL(\Lag,\Phi) + \sum_{(v\in \fr \V)\prec \alpha} d\big( \Theta^v_\alpha (\Lag) + \lambda_\alpha \circ \Theta^v_ \alpha (\Phi)\big),
$$
from which, if $D^{i} \in \mathfrak{X}\big(\Gamma(\V_0,Y)\times \textrm{Map}(\V,\alglie^*)\big)$, $i=1,2$, are extensions of $\big( (\delta y)^{i},(\delta\lambda)^{i}\big)$  verifying the Jacobi equation \eqref{ecuacion-jacobi}, then  by Proposition  \ref{prop:deriv-accion-ext} we have:
\begin{align*}
0 = &  \Big(\big(d\mathbb{A}_\V(\Lag,\Phi)\big)(D^1,D^2)\Big)(y,\lambda) = \Big(  \big( d \EL (\Lag,\Phi)\big)(D^{1}, D^{2}) \Big)(y,\lambda) \\
&  +   \sum_{(v\in \fr \V)\prec \alpha} \Big( d\big( \Theta^v_\alpha (\Lag) + \lambda_\alpha \circ \Theta^v_ \alpha (\Phi)\big)\Big)_ {(y,\lambda)}  \big( \big( (\delta y)^1, (\delta\lambda)^1\big), \big( (\delta y)^2, (\delta\lambda)^2\big)\big) .
\end{align*}

Now, applying the  Cartan formula to the first term of the previous sum, and taking into account the Jacobi equation \eqref{ecuacion-jacobi} and Euler-Lagrange equations $\big(\EL (\Lag,\Phi) \big)_{(y,\lambda)}=0$, we obtain:
% cambio
\begin{align*}
 \Big(\big( d  \EL & (\Lag,\Phi)\big)(D^{1}, D^{2}) \Big)(y,\lambda)  \\
 = &  \Big(D^{1}\big(\EL (\Lag,\Phi) (D^{2})\big)\Big)(y,\lambda) -   \Big(D^{2}\big(\EL (\Lag,\Phi) (D^{1})\big)\Big)(y,\lambda) - \big( \EL(\Lag,\Phi)([D^{1},D^{2}])\big) (y,\lambda)\\
 = & \big( L_{D^{1}} \EL(\Lag,\Phi)\big)_{(y,\lambda)} (D^{2}_{(y,\lambda)}) - \big( L_{D^{2}} \EL(\Lag,\Phi)\big)_{(y,\lambda)} (D^{1}_{(y,\lambda)})  + \big( \EL(\Lag,\Phi)([D^{1},D^{2}])\big) (y,\lambda)\\
  = & 0.
\end{align*}

So, the result is concluded. 
\end{proof}

\section{Application to Euler-Poincar\'e reduction in discrete field theory}\label{sec:euler-poincare}

In its simplest original version, this discrete reduction problem arises when trying to solve an unconstrained discrete problem for a principal bundle over the standard simplicial complex of $\R^2$ with Lagrangian density invariant by the action of the structural group of the bundle (see for example \cite{vank} and the references cited therein).

More precisely, following \cite{vank}, let $V$ be the simplicial complex with vertices $V_0=\big\{(i,j)\in \Z\times\Z\big\}$, oriented edges $V_1=\big\{ [(i,j),(i+1,j)], [(i,j),(i,j+1)], \, i,j\in \Z\big\}$, and oriented faces $V_2=\big\{\Delta_{ij}=[(i,j),(i+1,j),(i,j+1)],\, i,j\in \Z\big\}$. In this case, $V_2$ is identified with $V_0$ by the bijection $\Delta_{ij}\mapsto ( i,j)$  which we will assume in what follows. Given a Lie group $G$, let $P=G\times V_0 \to V_0$ be the (left) trivial principal bundle, and we consider a Lagrangian density $\Lag:(J^1P=G\times G\times G\times V_0  \to V_0) \to \R$ invariant by the diagonal action of $G$ on $J^1P: g\big((g_{ij}, g_{i+1j},g_{ij+1}),(i,j)\big) =  \big((gg_{ij}, gg_{i+1j},gg_{ij+1}),(i,j)\big)$. Identifying $J^1P/G$ with the fiber bundle $Y=G\times G \times V_0 \to V_0$ by the rule $G\big((g_{ij}, g_{i+1j},g_{ij+1}),(i,j)\big)  = \big((g_{ij}^{-1}g_{i+1j},g_{ij}^{-1}g_{ij+1}),(i,j)\big) $, the Lagrangian density $\Lag$ is projected onto a Lagrangian density $\lag :Y=J^1P/G \to \R$ by the projection (reduction map)
\begin{align*}
\pi: J^1P \hspace{1em} \longrightarrow & \hspace{1em} Y= J^1P/G\\
\big( (g_{ij},g_{i+1j},g_{ij+1}),(i,j)\big) \mapsto & \big( (u_{ij}= g^{-1}_{ij} g_{i+1j}, v_{ij}= g^{-1}_{ij} g_{ij+1}), (i,j)\big)
\end{align*}

%\begin{equation*}
%\xymatrix@R=1mm{
%J^1P \ar[r]^{\pi} & Y= J^1P/G\\
%\big( (g_{ij},g_{i+1j},g_{ij+1}),(i,j)\big) \ar@{|->}[r] & \big( (u_{ij}= g^{-1}_{ij} g_{i+1j}, v_{ij}= g^{-1}_{ij} g_{ij+1}), (i,j)\big)
%}
%\end{equation*}

On the other hand, the sections $g:(i,j)\in V_0 \mapsto \big(g_{ij},(i,j)\big)$ of $P$ are projected by $\pi\circ j^1$ in the sections $ y:(i,j)\in V_0 \mapsto  \big( (u_{ij}= g^{-1}_{ij} g_{i+1j}, v_{ij}= g^{-1}_{ij} g_{ij+1}), (i,j)\big)$ of $Y$ satisfying the constraint $\Phi(j^1y)=e$ where $\Phi:J^1Y\to G$ on $(J^1Y)_{(ij)}$ is given by the formula
\begin{equation}\label{ligadura-reduccion}
\Phi_{ij}\big( (u_{ij},v_{ij}, u_{i+1j},v_{i+1j}, u_{ij+1},v_{ij+1}),(i,j)\big) = u_{ij}v_{i+1j}u_{ij+1}^{-1}v_{ij}^{-1}.
\end{equation}

Similarly, the infinitesimal variations $\delta g:(i,j)\in V_0 \mapsto \big(\delta g_{ij},(i,j)\big)$  of each section $g\in \Gamma(V_0,P)$ are projected by $\pi\circ j^1$ in the infinitesimal variations $\delta y:(i,j)\in V_0\mapsto \big( (\delta u_{ij},\delta v_{ij}),(i,j)\big)$ of the section $y\in \Gamma(V_0,Y)$ projection of $g$, where:
% cambio
\begin{equation}\label{var-inf-reducidas}
\begin{split}
\delta u_{ij} = & (L_{u_{ij}})_*(\theta_{i+1j}) - (R_{u_{ij}})_*(\theta_{ij})\in T_{u_{ij}} G,\\
\delta v_{ij} = &(L_{v_{ij}})_*(\theta_{ij+1}) - (R_{v_{ij}})_*(\theta_{ij})\in T_{v_{ij}} G,
\end{split}
\end{equation}
being $\theta_{ij}= (L_{g_{ij}^{-1}})_* ( \delta g_{ij}) \in \alglie$.

Thus we have a  constrained variational problem on $Y=J^1P/G$ (reduced problem) whose action functional with respect to a finite set of faces $\V\subset V_2$ ($\V\equiv \U\subset V_0$ by the bijection $\Delta_{ij}\mapsto (i,j)$) is:
$$
\mathcal{A}_{\U}(\lag) : y\in \Gamma(\U_0,Y) \mapsto \sum_{(i,j)\in \U} \lag \big( (u_{ij}, v_{ij}), (i,j)\big),
$$
where $\U_0$ is the (finite) set of adherent vertices to $\V$.

Under these conditions, we will say that a section $y\in \Gamma(\U_0,Y)$ is critical for the reduced problem if $(d\mathcal{A}_\U(\lag))_y$ is zero over the infinitesimal variations \eqref{var-inf-reducidas} such that $\theta_{ij}$ vanishes at the frontier of $\U$.

\begin{rmk}
By similarity with the continuous case \cite{reduccion-cont}, the sections $y\in \Gamma(V_0,Y)$ of the reduced bundle $Y=J^1P/G$ are identified with the connections of the principal bundle $P$, the constraint morphism \eqref{ligadura-reduccion} defines the curvature of the connection, and the reduced infinitesimal variations \eqref{var-inf-reducidas} are interpreted as the action on the connections of the infinitesimal gauge  transformations of $P$, $\Gamma(V_0,\alglie\times V_0\to V_0)$. 
\end{rmk}

Under this approach, the main result of \cite{vank} is as follows.

\begin{thm}[Reduction]\label{thm:reduccion}
Let $\Lag$ be a $G$-invariant Lagrangian density on $J^1P$ and consider the reduced Lagrangian density $\lag$ on $Y=J^1P/G$. Consider a section $g\in \Gamma(\U_0,P)$ and let $y\in \Gamma(\U_0,Y)$ the induced reduced section. Then the  following conditions are equivalent:
\begin{enumerate}
\item $g$ is a solution of the discrete Euler-Lagrange equations for $\Lag$ in the interior points of $\U$.
\item $g$ is critical for arbitrary variations $\delta g$ that vanishes in the frontier of $\U$.

\item the reduced section $y$ is a solution of the discrete Euler-Poincar\'e equations in the interior points of $\U$:
\begin{equation}\label{ecuaciones-EP}
R^*_{u_{ij}}\big(d\lag(\cdot, v_{ij})\big)_e - L^*_{u_{i-1j}}\big(d\lag(\cdot, v_{i-1j})\big)_e
+ R^*_{v_{ij}}\big(d\lag(u_{ij},\cdot)\big)_e - L^*_{v_{ij-1}}\big(d\lag(u_{ij-1},\cdot)\big)_e =0.
\end{equation}
\item the reduced section $y$ is critical for infinitesimal variations $\delta y$ \eqref{var-inf-reducidas} where $\theta_{ij}$ are arbitrary variations that vanishes in the frontier of $\U$.
\end{enumerate}
\end{thm}

From Theorem \ref{thm:reduccion}, a solution $g\in \Gamma(\U_0,P)$ of the unreduced Euler-Lagrange equations gives rise to a reduced section $y\in \Gamma(\U_0,Y)$ verifying the constraint $\Phi(j^1y)=e$.

Reciprocally \cite{vank}:

\begin{thm}[Reconstruction]%\label{thm:reconstruccion}
Let $y\in \Gamma(\U_0,Y)$ be a solution of the discrete Euler-Poincar\'e equations \eqref{ecuaciones-EP} verifying the constraint $\Phi(j^1y)=e$. Then there exists a solution $g\in \Gamma(\U_0,P)$ of the unreduced Euler-Lagrange equations that projects over $y$. In this case, $g$ is uniquely determined up to right translation by an element of $G$.
\end{thm}

The Euler-Poincar\'e reduction just summarized provides an interesting example of a  discrete Lagrangian problem valued in a Lie group taking as problem data the reduced Lagrangian density $\lag : Y=J^1P/G \to \R$ and the constraint morphism $\Phi:J^1Y \to G$ given by formula \eqref{ligadura-reduccion}.

In particular, the solutions of this problem will be solutions of the discrete Euler-Poincar\'e equations \eqref{ecuaciones-EP}, which has the nice consequence of being able to apply to the discrete Euler-Poincar\'e reduction the formalism that we have developed in this paper for Lagrange problems.

The procedure we are going to follow is to eliminate the Lagrange multiplier from  equations \eqref{EL-prob-ligado} and observe that the result of such elimination satisfies the Euler-Poincar\'e equations \eqref{ecuaciones-EP}.

Indeed. Equations \eqref{EL-prob-ligado} are in this case:
\begin{multline}\label{EL-ligado-en-reduccion}
%\begin{split}
(d\lag)_{y_{ij}}+  \lambda_{\Delta_{ij}}\circ \big(\Theta_{\Delta_{ij}}^{ij}(\Phi)\big)_{j^1y(\Delta_{ij})} + \lambda_{\Delta_{i-1j}}\circ \big(\Theta_{\Delta_{i-1j}}^{ij}(\Phi)\big)_{j^1y(\Delta_{ij})} \\
 +  \lambda_{\Delta_{ij-1}}\circ \big(\Theta_{\Delta_{ij-1}}^{ij}(\Phi)\big)_{j^1y(\Delta_{ij})} = 0
%\end{split}
\end{multline}
where, because $y$ is admissible, the 1-forms $\Theta_{\Delta_{ij}}^{ij}(\Phi)$ are calculated along $j^1y$ by the formula:
\begin{equation}\label{dphi}
\big(d\Phi_{\Delta_{ij}}\big)_{j^1y(\Delta_{ij})} =  \big(\Theta_{\Delta_{ij}}^{ij}(\Phi)\big)_{j^1y(\Delta_{ij})} +  \big(\Theta_{\Delta_{ij}}^{i+1j}(\Phi)\big)_{j^1y(\Delta_{ij})} +  \big(\Theta_{\Delta_{ij}}^{ij+1}(\Phi)\big)_{j^1y(\Delta_{ij})}.
\end{equation}

Applying \eqref{EL-ligado-en-reduccion} to $\delta y_{ij}= (\delta u_{ij},\delta v_{ij})$  these equations are equivalent to the system:
 \begin{equation}\label{sist-ec-EL-reducido}
\left\{
\begin{aligned}
d\lag(\cdot, v_{ij}) (\delta u_{ij}) + & \lambda_{\Delta_{ij}} \circ \big(\Theta_{\Delta_{ij}}^{ij}(\Phi)\big)_{j^1y(\Delta_{ij})} (\delta u_{ij}) 
+ \lambda_{\Delta_{i-1j}}\circ \big(\Theta_{\Delta_{i-1j}}^{ij}(\Phi)\big)_{j^1y(\Delta_{ij})} (\delta u_{ij}) \\
+ & \lambda_{\Delta_{ij-1}}\circ \big(\Theta_{\Delta_{ij-1}}^{ij}(\Phi)\big)_{j^1y(\Delta_{ij})} (\delta u_{ij})   =0 \\
d\lag(u_{ij}, \cdot ) (\delta v_{ij}) + & \lambda_{\Delta_{ij}}\circ \big(\Theta_{\Delta_{ij}}^{ij}(\Phi)\big)_{j^1y(\Delta_{ij})} (\delta v_{ij}) 
+ \lambda_{\Delta_{i-1j}}\circ \big(\Theta_{\Delta_{i-1j}}^{ij}(\Phi)\big)_{j^1y(\Delta_{ij})} (\delta v_{ij}) \\
+ & \lambda_{\Delta_{ij-1}}\circ \big(\Theta_{\Delta_{ij-1}}^{ij}(\Phi)\big)_{j^1y(\Delta_{ij})} (\delta v_{ij})   =0 
\end{aligned}
\right.
\end{equation}

On the other hand, by the expression \eqref{ligadura-reduccion} of the constraint morphism, the only nonzero components of $(d\Phi_{\Delta_{ij}})_{j^1y(\Delta_{ij})}$ are those resulting from applying this 1-form to $\delta u_{ij}$, $\delta v_{ij}$, $\delta u_{ij+1}$, $\delta v_{i+1j}$ which, taking into account the constraint condition of $y_{ij}=(u_{ij},v_{ij})$, are computed as follows:
\begin{align*}
(d\Phi_{\Delta_{ij}})_{j^1y(\Delta_{ij})} & (\delta u_{ij}) =  
\delta u_{ij} v_{i+1j} u^{-1}_{ij+1} v^{-1}_{ij} = \delta u_{ij} u^{-1}_{ij} = (R_{u^{-1}_{ij}})_* (\delta u_{ij}) ,\\
(d\Phi_{\Delta_{ij}})_{j^1y(\Delta_{ij})} & (\delta v_{ij}) = 
u_{ij} v_{i+1j} u^{-1}_{ij+1}( -v^{-1}_{ij}\delta v_{ij} v^{-1}_{ij})  = - (R_{v^{-1}_{ij}})_* (\delta v_{ij}) ,\\
(d\Phi_{\Delta_{ij}})_{j^1y(\Delta_{ij})} & (\delta u_{ij+1}) = 
u_{ij} v_{i+1j}(- u^{-1}_{ij+1} \delta u_{ij+1} u^{-1}_{ij+1}) v^{-1}_{ij} = - Ad_{v_{ij}} (\delta u_{ij+1} u^{-1}_{ij+1}) ,\\
(d\Phi_{\Delta_{ij}})_{j^1y(\Delta_{ij})} & (\delta v_{i+1j}) = 
u_{ij} \delta v_{i+1j}u^{-1}_{ij+1}  v^{-1}_{ij} = u_{ij} \delta v_{i+1j} v^{-1}_{i+1j}  u^{-1}_{ij} =  Ad_{u_{ij}} (\delta v_{i+1j} v^{-1}_{i+1j}) .
\end{align*}

From here, formula \eqref{dphi} allows to calculate all the terms of the system of equations \eqref{sist-ec-EL-reducido} obtaining:
$$
\left\{
\begin{aligned}
d\lag(\cdot, v_{ij}) (\delta u_{ij}) +  \lambda_{\Delta_{ij}} \circ(R_{u^{-1}_{ij}})_* (\delta u_{ij}) - \lambda_{\Delta_{ij-1}}\circ Ad_{v_{ij-1}}(R_{u^{-1}_{ij}})_* (\delta u_{ij}) & = 0 \\
d\lag(u_{ij}, \cdot ) (\delta v_{ij}) -  \lambda_{\Delta_{ij}} \circ(R_{v^{-1}_{ij}})_* (\delta v_{ij}) + \lambda_{\Delta_{i-1j}}\circ Ad_{u_{i-1j}}(R_{v^{-1}_{ij}})_* (\delta v_{ij})  & =0 
\end{aligned}
\right.
$$

By transposition and translating  to the identity element $e\in G$ by $R_{u_{ij}}$  and $R_{v_{ij}}$ respectively we obtain the system of equations:
 \begin{equation}\label{sist-ec-EL-reducido-resuelto}
\left\{
\begin{aligned}
(R_{u_{ij}})^*\big(d\lag(\cdot, v_{ij})\big) +  \lambda_{\Delta_{ij}} - Ad^*_{v_{ij-1}}  \lambda_{\Delta_{ij-1}}   & =0 \\
(R_{v_{ij}})^*\big(d\lag( u_{ij}, \cdot )\big) -  \lambda_{\Delta_{ij}} + Ad^*_{u_{i-1j}}  \lambda_{\Delta_{i-1j}}   & =0
\end{aligned}
\right.
\end{equation}

Under these conditions we have the following:

\begin{thm}
If $(y,\lambda)\in \Gamma(\U_0,Y)\times \textrm{Map}(\U,\alglie^*)$ is a section-mapping, with $y$ admissible, solution of  equations \eqref{sist-ec-EL-reducido-resuelto} of the Lagrange problem with Lagrangian density $\lag :Y= J^1P/G \to \R$ and constraint morphism $\Phi:J^1Y\to G$ given by  formula \eqref{ligadura-reduccion}, then its component $y$ is a solution of the Euler-Poincar\'e equations \eqref{ecuaciones-EP}.
\end{thm}

\begin{proof}
Passing the multipliers to the second member in \eqref{sist-ec-EL-reducido-resuelto} we have:
\begin{equation}\label{sist-ec-EL-reducido-despejado}
\left\{
\begin{aligned}
(R_{u_{ij}})^*\big(d\lag(\cdot, v_{ij})\big) = & -  \lambda_{\Delta_{ij}} + Ad^*_{v_{ij-1}}  \lambda_{\Delta_{ij-1}}   \\
(R_{v_{ij}})^*\big(d\lag( u_{ij}, \cdot )\big) = &   \lambda_{\Delta_{ij}} - Ad^*_{u_{i-1j}}  \lambda_{\Delta_{i-1j}}   
\end{aligned}
\right.
\end{equation}

% break
Subtracting $Ad^*_{u_{i-1j}}(R_{u_{i-1j}})^* \big(d\lag (\cdot, v_{i-1j})\big)$ in each member of the first equation and \break $Ad^*_{v_{ij-1}}(R_{v_{ij-1}})^*  \big(d\lag ( u_{ij-1},\cdot)\big)$ in the second one, and then adding both expressions then the left-hand side of the new equality will be:
\begin{align*}
(R_{u_{ij}})^*\big(d\lag(\cdot, v_{ij})\big) 
- & Ad^*_{u_{i-1j}}(R_{u_{i-1j}})^* \big(d\lag (\cdot, v_{i-1j})\big) 
+(R_{v_{ij}})^*\big(d\lag( u_{ij}, \cdot )\big) \\
- & Ad^*_{v_{ij-1}}(R_{v_{ij-1}})^* \big(d\lag ( u_{ij-1},\cdot)\big)
\end{align*}
which, taking into account the identities $Ad^*_{u_{i-1j}}R^*_{u_{i-1j}}=L^* _{u_{i-1j}}$ and $Ad^*_{v_{ij-1}}R^*_{v_{ij-1}}=L^* _{v_{ij-1}}$, is precisely the first member of the Euler-Poincar\'e equations \eqref{ecuaciones-EP}.

As for the right-hand side of this new equation, applying again \eqref{sist-ec-EL-reducido-despejado} we get:
\begin{align*}
\Big( & -  \lambda_{\Delta_{ij}} + Ad^*_{v_{ij-1}}  \lambda_{\Delta_{ij-1}} - Ad^*_{u_{i-1j}} \big( -\lambda_{\Delta_{i-1j}} + Ad^*_{v_{i-1j-1}}\lambda_{\Delta_{i-1j-1}}\big)  \Big) \\
& + \Big(  \lambda_{\Delta_{ij}} - Ad^*_{u_{i-1j}}  \lambda_{\Delta_{i-1j}} - Ad^*_{v_{ij-1}} \big(  \lambda_{\Delta_{ij-1}} - Ad^*_{u_{i-1j-1}}\lambda_{\Delta_{i-1j-1}} \big) \Big)\\
 =  & (Ad^*_{v_{ij-1}}\circ Ad^*_{u_{i-1j-1}} - Ad^*_{u_{i-1j}}\circ Ad^*_{v_{i-1j-1}} )\lambda_{\Delta_{i-1j-1}}
\end{align*}

However, since the section $y$ is admissible, it satisfies the constraint
$$ u_{i-1j-1}v_{ij-1}= v_{i-1j-1}u_{i-1j}, $$
and from here we obtain:
\begin{align*}
Ad^*_{v_{ij-1}}\circ Ad^*_{u_{i-1j-1}} = & (Ad_{u_{i-1j-1}}\circ Ad_{v_{ij-1}})^*
= Ad^*_{u_{i-1j-1}v_{ij-1}} =  Ad^*_{ v_{i-1j-1}u_{i-1j}} \\
= & Ad^*_{u_{i-1j}}\circ Ad^*_{v_{i-1j-1}}.
\end{align*}

Thus, this right-hand side vanishes, and the result can be concluded. 
\end{proof}

Reciprocally, in an  analogous way to the continuous case \cite{reduccion-cont}, the following result is verified.

\begin{thm}
If $y\in \Gamma(\U_0,Y)$ is an admissible section solution of the Euler-Poincar\'e equations \eqref{ecuaciones-EP}, then locally there exists a non-unique mapping $\lambda \in \textrm{Map}(\U,\alglie^*)$ such that $(y,\lambda)$ is a solution of  equations \eqref{sist-ec-EL-reducido-resuelto} of the Lagrange problem.
\end{thm}

\begin{proof}
For each vertex $(i,j)\in \interior \U$,  choosing arbitrarily $\lambda_{\Delta_{ij}}$, you can solve  uniquely $\lambda_{\Delta_{i-1j}}$ and $\lambda_{\Delta_{ij-1}}$ in the system of equations \eqref{sist-ec-EL-reducido-resuelto}  over the spherical neighborhood  $S_{(i,j)}=\{\Delta_{ij},\Delta_{i-1j},\Delta_{ij-1}\}$ of $(i,j)$, due to the isomorphisms $Ad^*_{u_{i-1j}}$ and $Ad^*_{v_{ij-1}}$ respectively. 
\end{proof}

\section{Example: discrete harmonic mappings from the discrete plane into the Lie group  \texorpdfstring{$SO(n)$}{SO(n)} }\label{sec:ejemplo} 

For this example, $P=SO(n)\times V_0 \to V_0$ and the reduced Lagrangian density $\lag: (Y= J^1P/SO(n) = SO(n)\times SO(n)\times V_0 \to V_0) \to \R$ is given by the formula
\begin{equation}\label{lag-ejemplo}
\lag\big( (u_{ij},v_{ij}),(i,j)\big) = \operatorname{tr} u_{ij} +  \operatorname{tr} v_{ij}
\end{equation}
where $SO(n)$ is considered immersed in the standard way in the space of the square matrices $M_n(\R)$ \cite{vank}.

If $(a_{kl})$ and $(b_{kl})$ are two copies of the usual local coordinates of $M_n(\R)$ coordinating, respectively, the components $u_{ij}\in SO(n)\subset M_n(\R)$ and $v_{ij}\in SO(n)\subset M_n(\R)$, then the Lagrangian density \eqref{lag-ejemplo} is expressed as:
\begin{equation}\label{lag-ejemplo-coord}
\lag = \sum_{i'=1}^n a_{i'i'}(u_{ij}) + b_{i'i'}(v_{ij}).
\end{equation}

On the other hand, the Lie algebra of $SO(n)$ has as basis the tangent vectors $E_{kl}=\big(\parcial{}{a_{kl}}-\parcial{}{a_{lk}}\big)_e$, $k<l$, for the coordinates $(a_{kl})$ and $E_{kl}=\big(\parcial{}{b_{kl}}-\parcial{}{b_{lk}}\big)_e$, $k<l$, for the coordinates $(b_{kl})$.

Under these conditions, the four terms of the Euler-Poincar\'e equations \eqref{ecuaciones-EP} can be calculated as follows.

Given $u_{ij}\in SO(n)$, the equations of the translation $R_{u_{ij}}: a\in SO(n) \mapsto au_{ij}$ in the coordinates $(a_{kl})$ are: $a_{st}\circ R_{u_{ij}}= \sum_v a_{vt}(u_{ij})a_{sv}$ from where:
\begin{align*}
\big( (R_{u_{ij}})_{*,e} E_{kl}\big) a_{st} = & E_{kl} (a_{st}\circ R_{u_{ij}}) = \big(\parcial{}{a_{kl}}-\parcial{}{a_{lk}}\big)_e \big( \sum_v a_{vt}(u_{ij})a_{sv}\big) \\
 = & \delta_{kj}a_{lt}(u_{ij}) - \delta_{ls}a_{kt}(u_{ij}).
\end{align*}

From this it follows:
\begin{equation}\label{traslacion-izqda-ejemplo}
\big( (R_{u_{ij}})_{*,e} E_{kl}\big)  = \sum_{j'=1}^n a_{lj'}(u_{ij})\parcial{}{a_{kj'}} - a_{kj'}(u_{ij})\parcial{}{a_{lj'}}.
\end{equation}

Then, taking into account \eqref{lag-ejemplo-coord} and \eqref {traslacion-izqda-ejemplo} we  have:
\begin{align*}
\big( R^*_{u_{ij}} d\lag(\cdot, v_{ij})\big)_e  E_{kl}  = & \big( d\lag(\cdot, v_{ij})\big)_{u_{ij}}  \big( (R_{u_{ij}})_{*,e} E_{kl}\big) =  \big( (R_{u_{ij}})_{*,e} E_{kl}\big) \lag(\cdot, v_{ij})\\
= & \Big(\sum_{j'=1}^n a_{lj'}(u_{ij})\parcial{}{a_{kj'}} - a_{kj'}(u_{ij})\parcial{}{a_{lj'}} \Big)\big(\sum_{i'=1}^n a_{i'i'}+ b_{i'i'}\big) \\
= & a_{lk}(u_{ij}) -  a_{kl}(u_{ij}).
\end{align*}

Proceeding analogously with the other terms of equations \eqref{ecuaciones-EP} we obtain:
\begin{align*}
\big( & L^*_{u_{i-1j}} d\lag(\cdot, v_{i-1j})\big)_e  E_{kl}  =  a_{lk}(u_{i-1j}) -  a_{kl}(u_{i-1j}), \\
\big( & R^*_{v_{ij}} d\lag( u_{ij},\cdot)\big)_e  E_{kl}  =  b_{lk}(v_{ij}) -  b_{kl}(v_{ij}), \\
\big( & L^*_{v_{ij-1}} d\lag( u_{ij-1}, \cdot)\big)_e  E_{kl}  =   b_{lk}(v_{ij-1}) -  b_{kl}(v_{ij-1}).
\end{align*}

Then,  equations \eqref{ecuaciones-EP}  applied to the elements $E_{kl}$, $k<l$ of the Lie algebra of $SO(n)$ are:
\begin{align*}
\big( a_{lk}(u_{ij}) -  a_{kl}(u_{ij}) \big)
 - & \big( a_{lk}(u_{i-1j}) -  a_{kl}(u_{i-1j}) \big)
 +\big(  b_{lk}(v_{ij}) -  b_{kl}(v_{ij})\big) \\
 - & \big( b_{lk}(v_{ij-1}) -  b_{kl}(v_{ij-1}) \big)   =0
\end{align*}
 or equivalently, with the notation of $a^t$ for the transposition of $a$,
 $$
 a_{kl}(u^t_{ij}-u_{ij} + u_{i-1j}- u^t_{i-1j}) 
 + b_{kl}(v^t_{ij}-v_{ij} + v_{ij-1}- v^t_{ij-1}) =0.
 $$
 
 In other words, equations (23) are equivalent to the fact that, for all $k<l$, the $(k,l)$ element of the matrix $u^t_{ij}-u_{ij} + u_{i-1j}- u^t_{i-1j}$  coincides with the $(k,l)$ element of the matrix $-(v^t_{ij}-v_{ij} + v_{ij-1}- v^t_{ij-1})$. This, by virtue of the skew-symmetry of both matrices, is equivalent to this equality being verified for every $(k,l)$ element. Thus, the Euler-Poincar\'e equations \eqref{ecuaciones-EP} of our example are:
$$
u_{ij} + v_{ij}- u_{i-1j} -  v_{ij-1} = 
u^t_{ij} +v^t_{ij}  - u^t_{i-1j}- v^t_{ij-1},
$$
which  were obtained in \cite{vank} by another procedure.

Applying again the above calculations to  equations \eqref{sist-ec-EL-reducido-resuelto} of the corresponding problem of  Lagrange, they can be expressed as follows:
$$
\left\{
\begin{aligned}
-  \lambda_{\Delta_{ij}} + Ad^*_{v_{ij-1}}  \lambda_{\Delta_{ij-1}}   & =  \sum_{k<l} \big(a_{lk}(u_{ij}) -  a_{kl}(u_{ij}) \big) E^*_{kl}\\
 \lambda_{\Delta_{ij}} - Ad^*_{u_{i-1j}}  \lambda_{\Delta_{i-1j}}   & =\sum_{k<l} \big(b_{lk}(v_{ij}) -  b_{kl}(v_{ij}) \big) E^*_{kl}
\end{aligned}
\right.
$$

\end{document}